\documentclass{amsart}
\usepackage{graphicx}
\usepackage{amssymb}

\newtheorem{thm}{Theorem}[section]

\newtheorem{lem}[thm]{Lemma}
\newtheorem{prop}[thm]{Proposition}
\theoremstyle{definition}

\theoremstyle{remark}
\newtheorem{rem}[thm]{Remark}

\begin{document}

\title[Counterexample to Sobolev regularity]{Some counterexamples to Sobolev regularity for degenerate Monge-Amp\`{e}re equations}
\author{Connor Mooney}
\address{Department of Mathematics, UT Austin, Austin, TX 78712}
\email{\tt  cmooney@math.utexas.edu}


\begin{abstract}
We construct a counterexample to $W^{2,1}$ regularity for convex solutions to 
$$\det D^2u \leq 1, \quad u|_{\partial \Omega} = 0$$ 
in two dimensions. We also prove a result on the propagation of singularities in two dimensions that are logarithmically
slower than Lipschitz. This generalizes a classical result of Alexandrov and is optimal by example.
\end{abstract}
\maketitle

\section{Introduction}
In this paper we investigate the $W^{2,1}$ regularity of Alexandrov solutions to degenerate Monge-Amp\`{e}re equations of the form
\begin{equation}\label{DMA}
 \det D^2u(x) = \rho(x) \leq 1 \text{ in } \Omega, \quad u|_{\partial \Omega} = 0,
\end{equation}
where $\Omega$ is a bounded convex domain in $\mathbb{R}^n$.

In the case that $\rho$ also has a positive lower bound, $W^{2,1}$ estimates were first obtained by De Philippis and Figalli (\cite{DF2}).
They showed that $\Delta u \log^{k}(2 + \Delta u)$ is integrable for any $k$.
It was subsequently shown that $D^2u$ is in fact $L^{1+\epsilon}$ for some $\epsilon$ depending
on dimension and $\|1/\rho\|_{L^{\infty}(\Omega)}$ (see \cite{DFS}, \cite{Sch}). These estimates are optimal in light of two-dimensional examples due to Wang (\cite{W}) with the homogeneity
$$u(\lambda x_1, \lambda^{\alpha} x_2) = \lambda^{1+\alpha}u(x_1, x_2).$$

These estimates fail when $\rho$ degenerates. In three and higher dimensions, it is not hard to construct solutions 
to ($\ref{DMA}$) that have a Lipschitz singularity on part of a hyperplane, so the second derivatives concentrate (see Section \ref{HighDim}). 
However, in two dimensions, a classical result of Alexandrov (\cite{A}, see also \cite{FL}) shows that Lipschitz singularities of solutions to $\det D^2u \leq 1$ propagate to the boundary. Thus,
in two dimensions solutions to ($\ref{DMA}$) are $C^1$ and $D^2u$ has no jump part. However, this leaves open the possibility that $D^2u$ has
nonzero Cantor part.

The main result of this paper is the construction of a solution to ($\ref{DMA}$) in two dimensions that is not $W^{2,1}$. This negatively answers an open
problem stated in both (\cite{DF1}) and (\cite{F}), which was motivated by potential applications to the semigeostrophic equation. We also prove
that singularities that are logarithmically slower than Lipschitz propagate, which generalizes the theorem of Alexandrov and is optimal by example.

The $W^{2,1}$ estimates mentioned above have applications to the global existence of weak solutions to the semigeostrophic equation (\cite{ACDF1}, \cite{ACDF2}). 
In this context, the density $\rho$ solves a continuity equation that preserves $L^{\infty}$ bounds. This is the only regularity property of $\rho$ that is globally preserved, due to nonlinear
coupling between $\rho$ and the velocity field. It is therefore useful to obtain estimates that depend on $L^{\infty}$ bounds for $\rho$ but not on its regularity.

To apply the results in \cite{DF2} and \cite{DFS} one must assume that $\rho$ is supported in the whole space.
However, in physically interesting cases, the initial density is compactly supported. It is thus natural to ask what one can show about solutions
to ($\ref{DMA}$). Our construction shows that, even in two dimensions, one must rely more on the specific structure of the semigeostrophic
equation to obtain existence results for compactly supported initial data.

The idea of our construction is to start with a one-dimensional convex function of $x_2$ in the half-space $\{x_1 < 0\}$
whose second derivative has nontrivial Cantor part, and extend to a convex function on $\mathbb{R}^2$ which lifts from these values without generating too much Monge-Amp\`{e}re measure.
To accomplish this we start with a ``building block'' $v_1$ that agrees with $|x_2|$ in $\{|x_2| \geq (x_1)_+^{\alpha}\}$ for some $\alpha > 1$,
and in the cusp $\{|x_2| < (x_1)_+^{\alpha}\}$ grows with the homogeneity $$v_1(\lambda x_1,\lambda^{\alpha}x_2) = \lambda^{\alpha} v_1(x_1,x_2).$$
By superposing vertically-translated rescalings of (a smoothed version of) $v_1$ in a self-similar way, we obtain our example.

Our main theorem is:
\begin{thm}\label{Main}
 For all $n \geq 2$, there exist solutions to $(\ref{DMA})$ that are not $W^{2,1}$.
\end{thm}

\begin{rem}
 It is obvious that solutions to ($\ref{DMA}$) in one dimension are $C^{1,1}$.
\end{rem}

\begin{rem}
 In our examples, the support of $\rho$ is irregular. In particular, in the higher-dimensional examples, the support of $\rho$ is a cusp revolved around an axis,
 and in the two-dimensional example, the support of $\rho$ has a very irregular ``fractal'' geometry.

 In e.g. \cite{DS} and \cite{G} the authors obtain interesting regularity results when $\rho$ degenerates in a specific way, motivated by applications
 to prescribed Gauss curvature.
\end{rem}

Our second result concerns the behavior of solutions to $(\ref{DMA})$ near a single line segment.
Since Lipschitz singularities propagate, $D^2u$ cannot concentrate on a line segment.
(In our two-dimensional counterexample to $W^{2,1}$ regularity, $D^2u$ concentrates on a family of horizontal rays.)
On the other hand, by modifying an example in \cite{W} one can construct, for any $\epsilon > 0$, a solution to $(\ref{DMA})$
that grows like $|x_2|/|\log x_2|^{1 + \epsilon}$, with second derivatives not in $L\log^{1 + \epsilon} L$ (see Section \ref{Propagation}).

It is natural to ask whether one can take $\epsilon \leq 0$. We show that this is not possible. Indeed, we construct a family of barriers that agree with $|x_2|/|\log x_2|$ away from arbitrarily thin cusps
around the $x_1$ axis, where we can make the Monge-Amp\`{e}re measure as large as we like. By sliding these barriers we prove that singularities of the form $|x_2|/|\log x_2|$ propagate.
Our second theorem is:

\begin{thm}\label{LogarithmicPropagation}
 Assume that $u$ is convex on $\mathbb{R}^2$ and that $\det D^2u \leq 1$. Then if $u(0) = 0$ and
 $u \geq c|x_2|/|\log x_2|$ in a neighborhood of the origin for some $c > 0$, then $u$ vanishes on the $x_1$ axis.
\end{thm}

\begin{rem}
 Note that we assume the growth in a neighborhood of $0$. For a Lipschitz singularity it is enough to assume the growth at a point, which
 automatically extends to a neighborhood by convexity. (See e.g. \cite{FL} for a short proof that Lipschitz singularities propagate.)
\end{rem}

\begin{rem}
Theorem \ref{LogarithmicPropagation} shows, after taking a Legendre transform, that a solution to $\det D^2u \geq 1$ in two dimensions cannot separate from a tangent plane more slowly than $r^{2}e^{-\frac{1}{r}}$ in any fixed direction. 
This quantifies the classical result that such functions are strictly convex.
\end{rem}

The paper is organized as follows. In Section \ref{HighDim} we construct simple examples of solutions to ($\ref{DMA}$) in the case $n \geq 3$ which
have a Lipschitz singularity on a hyperplane.
In Section \ref{MainExample} we construct a solution to ($\ref{DMA}$) in two dimensions whose second derivatives have nontrivial Cantor part. This proves
Theorem \ref{Main}.
In Section \ref{Propagation} we first construct examples showing that Theorem \ref{LogarithmicPropagation} 
is optimal. We then construct barriers related to these examples. Finally, we use the barriers to prove Theorem \ref{LogarithmicPropagation}. 
\section{The Case $n \geq 3$}\label{HighDim}

In this section we construct simple examples of solutions to ($\ref{DMA}$) in three and higher dimensions that have a Lipschitz singularity on a hyperplane. 
Denote $x \in \mathbb{R}^n$ by $(x',x_n)$ and let $r = |x'|$. More precisely, we show:

\begin{prop}\label{HigherDimEx}
 In dimension $n \geq 3$, for any $\alpha \geq \frac{n}{n-2}$ there exists a solution to $(\ref{DMA})$ that is a positive multiple of $|x_n|$
 in $\{|x_n| \geq (r-1)_+^{\alpha}\}$.
\end{prop}
\begin{proof}
Let $h(r) = (r-1)_+$.  We search for a convex function $u = u(r,x_n)$ in  $\{|x_n| < h(r)^{\alpha}\}$, with $\alpha > 1$, that glues ``nicely'' across the boundary to $|x_n|$.
To that end we look for a function with the homogeneity 
$$u(1+\lambda t,\lambda^{\alpha} x_n) = \lambda^{\alpha}u(1 + t, x_n),$$ 
so that $u_n$ is invariant under the rescaling. Let
$$u(r,x_n) = \begin{cases}
                  h(r)^{\alpha} + h(r)^{-\alpha}x_n^2, \quad |x_n| < h(r)^{\alpha} \\
                  2|x_n|, \quad |x_n| \geq h(r)^{\alpha}.
                 \end{cases}
$$
Then $\nabla u$ is continuous on $\partial \{|x_n| < h(r)^{\alpha}\} \backslash \{r = 1, \, x_n = 0\}$. Furthermore,
$\partial u|_{\{r=1,\,x_n = 0\}}$ is the line segment between $\pm 2 e_n$, which has measure zero. Thus, in the Alexandrov sense, $\det D^2u$ can be computed piecewise. 
In the cylindrical coordinates $(r,x_n)$ one easily computes
$$\det D^2u = \begin{cases}
               \frac{2\alpha^{n-1}(\alpha-1)h(r)^{\alpha(n-2) - n}\left(1 - \left(\frac{x_n}{h(r)^{\alpha}}\right)^2\right)^{n-1}}{r^{n-2}},
               \quad |x_n| < h(r)^{\alpha} \\
               0, \quad |x_n| \geq h(r)^{\alpha}.
              \end{cases}
$$
For $\alpha \geq \frac{n}{n-2}$ the right hand side is locally bounded.
\end{proof}

\begin{rem}
 The bound on $\alpha$ can be understood by looking at the gradient map of $u$, which takes a ``ring'' of volume like $h(r)^{1+\alpha}$
 to a ``football'' of length of order $1$ and radius of order $h(r)^{\alpha - 1}$ (see Figure \ref{GradMap}). Then impose that it decreases volume.
\end{rem}

\begin{figure}
 \centering
    \includegraphics[scale=0.35]{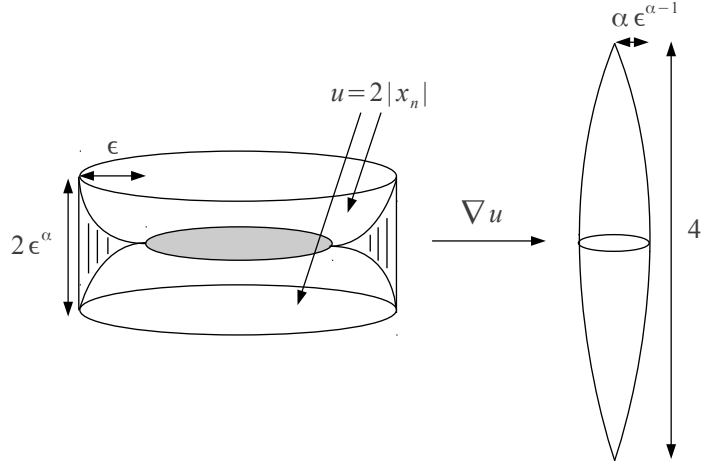}
 \caption{The gradient map of $u$ decreases volume if $\alpha \geq \frac{n}{n-2}$.}
\label{GradMap}
\end{figure}

\begin{rem}
 Observe that $\det D^2u$ grows like $\text{dist.}^{n-2-\frac{n}{\alpha}}$ from its zero set. This is in a sense optimal; if
 $\det D^2u < C|x_n|^{n-2}$ then one can modify Alexandrov's two-dimensional argument to show that the singularity has no extremal points.
\end{rem}

\section{The Case $n = 2$}\label{MainExample}
In this section we prove the main Theorem \ref{Main}. We construct our example in several steps.

First, let $g(t)$ be a smooth, convex function such that $g(t) = 1/2$ for $t \leq 0$ and $g(t) = t^{\alpha}$ for $t \geq 1$, where
$\alpha > 1$. Then define
$$v_1(x_1,x_2) = \begin{cases}
                  g(x_1) + \frac{1}{g(x_1)}x_2^2, \quad |x_2| < g(x_1), \\
                  2|x_2|, \quad |x_2| \geq g(x_1). 
                 \end{cases}
$$
It is easy to check that $v_1$ is a $C^{1,1}$ convex function, and in the Alexandrov sense,
$$\det D^2v_1(x_1,x_2) = \begin{cases}
                 2\frac{g''(x_1)}{g(x_1)}\left(1 - \frac{x_2^2}{g(x_1)^2}\right), \quad |x_2| < g(x_1), \\
                 0, \quad |x_2| \geq g(x_1).
                \end{cases}
$$
In particular, $\det D^2 v_1$ is bounded, and decays like $x_1^{-2}$ for $x_1$ large. Let $v_{\lambda}$ be the rescalings defined by
$$v_{\lambda}(x_1,x_2) = \frac{1}{\lambda^{1 + \alpha}}v_1(\lambda x_1, \lambda^{\alpha} x_2).$$
Observe that
$$\det D^2v_{\lambda}(x_1,x_2) = \det D^2v_1(\lambda x_1,\lambda^{\alpha} x_2),$$
and we have
\begin{equation}\label{Invariance}
 v_{\lambda} = \frac{1}{\lambda}(x_1^{\alpha} + x_1^{-\alpha}x_2^2) \quad \text{ in } \quad \{x_1 \geq \lambda^{-1}\} \cap \{|x_2| \leq x_1^{\alpha}\}.
\end{equation}

In the following key lemma we show that any superposition of $\lambda$ vertically-translated copies of $v_{\lambda}$ has bounded Monge-Amp\`{e}re measure in $\{x_1 > 1/2\}$, and separates from
its tangent planes when we step away from the $x_2$ axis.

\begin{lem}\label{KeyLemma}
 Let $\{x_{2,i}\}_{i = 1}^N$ be fixed numbers with $|x_{2,i}| \leq 1$ for all $i$, where $N$ is any positive integer. Let
 $$w(x_1,x_2) = \sum_{i = 1}^N v_{N}(x_1,x_2-x_{2,i}).$$
 Then
\begin{equation}\label{DetBound}
  \det D^2w < C(\alpha) \quad \text{ in } \quad \{x_1 > 1/2\}
\end{equation}
  for some $C(\alpha)$ independent of $N$ and the choice of $\{x_{2,i}\}$, and
\begin{equation}\label{Separation}
 w(2,x_2) > w(0,x_2) + \mu(\alpha) \quad \text { for all } \quad |x_2| < 1,
\end{equation}
  for some $\mu(\alpha) > 0$ independent of $N$ and the choice of $\{x_{2,i}\}$.
\end{lem}

\begin{proof}
{\bf Proof of (\ref{DetBound}):} Since $\det D^2v_1$ is bounded we may assume that $N \geq 2$. Let $p = (p_1,p_2) \in \{x_1 > 1/2\}$.
Since $w$ is $C^{1,1}$, the curves $p_2 = x_{2,i} \pm p_1^{\alpha}$ don't contribute anything to $\det D^2w$, so we may assume that
$p_2 \neq x_{2,i} \pm p_1^{\alpha}$ for any $i$. Then in a neighborhood of $p$, a subset
of $M \leq N$ of the translates are not linear, and all are linear if in addition $|p_2| > 1 + p_1^{\alpha}$. 
Up to relabeling the indices and subtracting a linear function of $x_2$, by (\ref{Invariance}) we can write
$$w = \frac{M}{N} \left(x_1^{\alpha} + x_1^{-\alpha}\left(x_2^2 - 2x_2\frac{\sum_{i = 1}^M x_{2,i}}{M} + \frac{\sum_{i = 1}^M x_{2,i}^2}{M}\right)\right)$$
in a neighborhood of $p$.
Since $|x_{2,i}| \leq 1$, one easily computes that
$$\det D^2w(p) \leq 2\alpha\frac{M^2}{N^2}p_1^{-2}\left(\alpha - 1 + (\alpha + 1)p_1^{-2\alpha} (p_2^2 + 2|p_2| + 1)\right),$$
and $\det D^2w(p) = 0$ if $|p_2| > 1 + p_1^{\alpha}$. We conclude that
$$\det D^2w(p) < C(\alpha),$$
where $C(\alpha)$ does not depend on $N$.

{\bf Proof of (\ref{Separation}):} Since $v_1$ is monotone increasing in the $e_1$ direction, we have for $|x_2| \leq 2$ that
$$v_1(2,x_2) - v_1(0,x_2) \geq v_1(2,x_2) - v_1(2^{1/\alpha},x_2) \geq 2^{-\alpha}(2^{\alpha} - 2)^2.$$
Since $\alpha > 1$, the lower bound $\mu := 2^{-\alpha}(2^{\alpha} - 2)^2$ is strictly positive.

By (\ref{Invariance}) the same argument gives
$$v_{N}(2,x_2) - v_{N}(0,x_2) > \frac{\mu}{N}$$
for $|x_2| \leq 2$. Finally, since $|x_{2,i}| \leq 1$, we have for $|x_2| < 1$ that
$$\sum_{i = 1}^N [v_N(2,x_2-x_{2,i}) - v_N(0,x_2-x_{2,i})] \geq \sum_{i = 1}^N \frac{\mu}{N} = \mu > 0,$$
completing the proof.
\end{proof}

We can now complete the construction. Roughly, at stage $k$ we superpose $2^{k+1}$ vertical translations of $v_{2^{k+1}}$, starting at the endpoints of the intervals 
removed up to the $k^{th}$ stage in the construction of the Cantor set.

\begin{proof}[{\bf Proof of Theorem \ref{Main}}]
Fix
$$\alpha := \frac{\log 3}{\log 2},$$ 
and define
$$u_1(x_1,x_2) = \sum_{i = 0}^3 v_4(x_1, x_2 - 1 + 2i/3).$$
Then $u_1$ is a piecewise linear function of $x_2$ outside of four equally spaced cusps in $\{x_1 > 0\}$ connected to thin strips in $\{x_1 < 0\}$ (see Figure \ref{Cantor1}).

\begin{figure}
 \centering
    \includegraphics[scale=0.35]{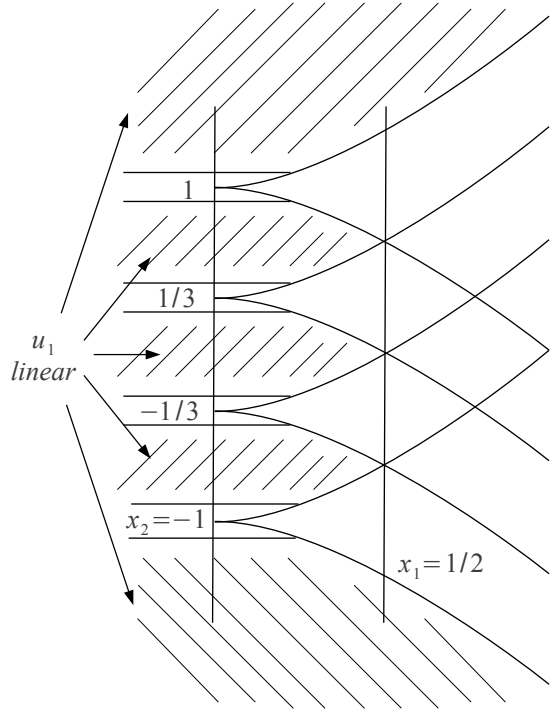}
 \caption{The function $u_1$ is a piecewise linear function of $x_2$ outside of the four equally spaced cusps between $x_2 = -1$ and $x_2 = 1$.}
\label{Cantor1}
\end{figure}

Define $u_k$ inductively by
$$u_{k+1}(x_1,x_2) = \frac{1}{2^{1 + \alpha}}\left[u_k(2x_1, 3(x_2 + 2/3)) + u_k(2x_1, 3(x_2-2/3))\right].$$

We first claim that $\det D^2u_k$ are uniformly bounded (in $k$) in $\{x_1 > 1/2\}$. Indeed, each $u_k$ is a sum of $2^{k+1}$ vertical translates
of $v_{2^{k+1}}$ by values in $[-1,1]$, so this follows from ($\ref{DetBound}$).

Next we show that $\det D^2u_k$ are uniformly bounded in $\mathbb{R}^2$. Note that $u_k$ are linear functions of $x_2$ in
$\{x_1 \leq 1\} \times \{|x_2| > 2\}$, so in $\{x_1 \leq 1/2\}$, the rescaled copies of $u_k$ in the definition of $u_{k+1}$ are linear
where the other is nontrivial (the determinants ``don't interact'', see Figure \ref{Cantor2}). Since the rescaling $2^{-(1+\alpha)}u_k(2x_1,3x_2)$ preserves Hessian determinants, we conclude that
$$\det D^2u_{k+1}|_{\{x_1 \leq 1/2\}} \leq \sup_{x_1 \geq 0} \det D^2u_k.$$
One easily checks that $\det D^2 u_1$ is bounded, so the claim follows by induction.

\begin{figure}
 \centering
    \includegraphics[scale=0.35]{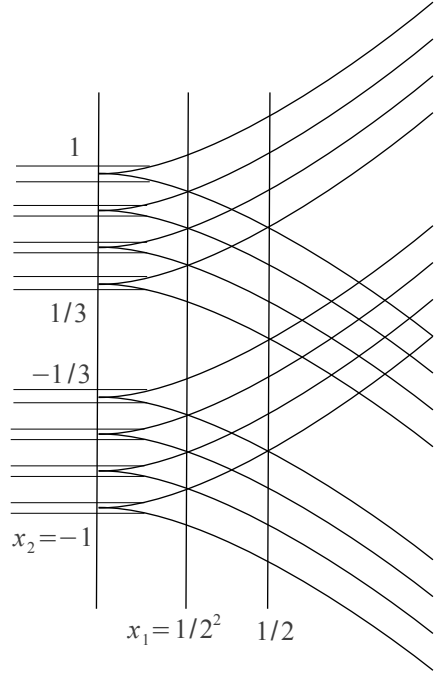}
 \caption{The function $u_2$ is obtained by superposing two rescaled copies of $u_1$, whose Hessians don't affect each other in $\{x_1 \leq 1/2\}$.}
 \label{Cantor2}
\end{figure}

Since $|v_{\lambda}|,\,|\nabla v_{\lambda}| < C\frac{R^{\alpha}}{\lambda}$ in $B_R$, the 
functions $u_k$ are locally uniformly Lipschitz and bounded and thus converge locally uniformly to some $u_{\infty}$. 
The right hand sides $\det D^2u_k$ converge weakly to $\det D^2u_{\infty}$ (see \cite{Gut}), so 
$$\det D^2u_{\infty} < \Lambda < \infty$$
in all of $\mathbb{R}^2$.

Finally, let 
$$u(x_1,x_2) = u_{\infty}((|x_1|-1)_+, x_2)$$
be the function obtained by translating $u_{\infty}$ to the right and reflecting over the $x_2$ axis. 

It is clear that $u$ is even in $x_1$ and $x_2$, and is a one-dimensional function $f(x_2)$ in the strip $\{|x_1| < 1\}$. It is easy to show
that $f'$ is the standard Cantor function (appropriately rescaled), so $f''$ has a nontrivial Cantor part. Indeed, $\partial_2u_k(0,\cdot)$ jumps by $2^{1-k}$ over
each of $2^{k+1}$ intervals of length $3^{-(k+1)}$ centered at the endpoints of the sets removed in the construction of the Cantor set. 
By (\ref{Separation}) we also have
$$u(\pm 2,0) > u(0,0) + \mu.$$
Since $u$ is even over both axes we conclude that
$$\{u < u(0,0) + \mu\} \subset [-2,2] \times [-C,C].$$
By convexity, $u$ has bounded sublevel sets, completing the proof.
\end{proof}
\section{A Propagation Result}\label{Propagation}
The second derivatives of a solution to $(\ref{DMA})$ cannot concentrate on a single line segment, since Lipschitz singularities propagate. (Compare to the example above,
where the second derivatives concentrate on a family of horizontal rays.)
In this section we investigate more closely how the second derivatives of a solution to $(\ref{DMA})$ can behave near a single line segment.

We first construct, for any $\epsilon > 0$, examples that grow from the origin like $|x_2|/|\log x_2|^{1+\epsilon}$,
with $D^2u$ not in $L\log^{1+\epsilon} L$. We then construct a family of barriers related to these examples in the case $\epsilon = 0$. Finally, we
use these barriers to prove that singularities of the form $|x_2|/|\log x_2|$ propagate.

\subsection{Examples that Grow Logarithmically Slower than Lipschitz}
\begin{prop}\label{LogExamples}
 For any $\alpha > 0$ there exists a solution to $(\ref{DMA})$ that vanishes at $0$ and lies above $c|x_2|/|\log x_2|^{1+1/\alpha}$, and
 whose Hessian is not in $L\log^{1 + 1/\alpha} L$.
\end{prop}
\begin{proof}
 Let $\Omega_1 = \{|x_2| < h(x_1)e^{-1/x_1^{\alpha}}\}$ for some positive even function $h$ to be determined. (By $x^{\gamma}$ we mean $|x|^{\gamma}$). 
 In $\Omega_1$, define
 $$u_0(x_1,x_2) = x_1^{\alpha + 1}e^{-1/x_1^{\alpha}} + x_1^{\alpha+1}e^{1/x_1^{\alpha}}x_2^2.$$
 We would like to glue this to a function of $x_2$ on $\Omega_2 = \mathbb{R}^2 \backslash \Omega_1$, which imposes the condition $\partial_1u_0 = 0$ on the boundary.
 Computing, we find that
 $$h^2(t) = \frac{1 + (\alpha + 1)t^{\alpha}/\alpha}{1- (\alpha + 1)t^{\alpha}/\alpha} = 1 + 2\frac{\alpha + 1}{\alpha} t^{\alpha} + O(t^{2\alpha}).$$ 
 In this way we ensure that $u_0$ glues in a $C^1$ manner across $\partial\Omega_1$ to some function $g(|x_2|)$ in $\Omega_2$ defined by
 $$g(h(t)e^{-1/t^{\alpha}}) = t^{\alpha + 1}(1 + h^2(t))e^{-1/t^{\alpha}}.$$
 The agreement of derivatives on $\partial \Omega_1$ gives
 $$g'(h(t)e^{-1/t^{\alpha}}) = 2t^{\alpha + 1}h(t),$$
 which upon differentiation and using the formula for $h$ gives
 $$g''(h(t) e^{-1/t^{\alpha}}) = 2(1 + 1/\alpha + o(1))e^{1/t^{\alpha}}t^{2\alpha + 1}.$$
 For $|z|$ small it follows that
 $$g''(z) \geq \frac{1}{|z||\log z|^{2 + 1/\alpha}},$$
 giving the non-integrability claimed (after, say, replacing $x_1$ by $(|x_1|-1)_+$). 

 It remains to show that $\det D^2 u_0$ is positive and bounded. One computes for 
 $$x_2^2 = s^2h(x_1)^2e^{-2/x_1^{\alpha}}, \quad s^2 < 1$$ that
 $$\det D^2 u_0(x_1,x_2) = 2\alpha^2\left\{(1 - s^2) + (\alpha + 1)x_1^{\alpha}(1+s^2)/\alpha\right\} + O(x_1^{2\alpha}),$$
 completing the proof.
\end{proof}

\subsection{Barriers}
We now construct barriers that agree with $|x_2|/|\log x_2|$ except for in a very thin cusp around the $x_1$ axis 
where the Monge-Amp\`{e}re measure is as large as we like. Let 
$$h_{\alpha}(t) = \begin{cases}
          0, \quad t \leq 0 \\
          \frac{1}{2}e^{-1/t^{\alpha}}, \quad t > 0
         \end{cases}
$$
where $\alpha > 0$ is large. Let $\Omega_{1,\alpha} = \{|x_2| < h_{\alpha}(x_1)\}$ be a thin cusp around the positive $x_1$ axis and let $\Omega_{2,\alpha}$ 
be its complement. Our barrier is
$$b_{\alpha}(x_1,x_2) = \begin{cases}
                x_1^{\alpha}e^{-1/x_1^{\alpha}} + x_1^{\alpha}e^{1/x_1^{\alpha}}x_2^2 \quad \text{ in } \quad \Omega_{1,\alpha}, \\
                \frac{5}{2}|x_2|/|\log 2 x_2| \quad \text{ in } \quad \Omega_{2,\alpha}.
               \end{cases}
$$

Note that $b_{\alpha}$ is convex and bounded by $1$ on $\Omega_{2,\alpha} \cap \{|x_2| < \frac{1}{4}\}$, and 
$b_{\alpha}$ is continuous across $\partial \Omega_{1,\alpha}$. 
Furthermore, on $\partial \Omega_{1,\alpha}$ one computes (from inside $\Omega_{1,\alpha}$) that
$$\partial_1b_{\alpha}(x_1,x_2) = \alpha e^{-1/x_1^{\alpha}}\left(\frac{3}{4}x_1^{-1} + \frac{5}{4}x_1^{\alpha - 1}\right) \geq 0,$$
so the derivatives have positive jumps across $\partial \Omega_{1,\alpha}$.

Denote $x_2^2e^{2/x_1^{\alpha}} = a$. One computes in $\Omega_{1,\alpha}$ (where $a \leq 1/4$) that
\begin{align*}
 \det D^2b_{\alpha} &= 2\alpha^2x_1^{-2}\left((1-a) + \frac{\alpha-1 + a(3\alpha + 1)}{\alpha}x_1^{\alpha}
 + \frac{\alpha-1 - a(\alpha + 1)}{\alpha} x_1^{2\alpha}\right) \\
 &\geq \frac{3}{2}\alpha^2x_1^{-2}.
\end{align*}

Finally, let $\Omega:= (-\infty, 1/2] \times [-1/4, 1/4].$ We conclude that $b_{\alpha}$ are convex in $\Omega$, with
$$\det D^2b_{\alpha} \geq 6 \alpha^2 \quad \text{ in } \quad \Omega_{1,\alpha} \cap \Omega,$$
and furthermore
$$b_{\alpha} < \frac{5}{4}2^{-\alpha}e^{-2^{\alpha}} \quad \text{ in } \quad \Omega_{1,\alpha} \cap \Omega.$$

\subsection{Propagation}
We prove Theorem \ref{LogarithmicPropagation} by sliding the barriers $b_{\alpha}$ from the right.

\begin{proof}[{\bf Proof of Theorem \ref{LogarithmicPropagation}}]
 By rescaling and multiplying by a constant, we may assume that
 $$u \geq \frac{5}{2}|x_2|/|\log 2x_2| \quad \text{ in } \quad \{|x_2| < 1/4\} \cap B_1,$$
 with $u(0) = 0$ and $\det D^2u < \Lambda$ for some large $\Lambda$. Choose $\alpha$ so large that $\alpha^2 > \Lambda$.
 Slide the barriers $b_{\alpha}(\cdot - te_1)$ from the right. Since $u \geq b_{\alpha}(\cdot - te_1)$ on $\partial(\Omega_{1,\alpha} + te_1) \cap \Omega$ for all
 $|t|$ small, it follows from the maximum principle that
 $$u(1/2,x_2) \leq b_{\alpha}(1/2,x_2)$$
 for some $(1/2,x_2) \in \overline{\Omega_{1,\alpha}} \cap \Omega$. (Indeed, if not, we can take $t = -\epsilon$ small and get a contradiction; see Figure \ref{Barrier}). 
 Taking $\alpha \rightarrow \infty$, we conclude that $u(e_1/2) = 0$.

 \begin{figure}
 \centering
    \includegraphics[scale=0.35]{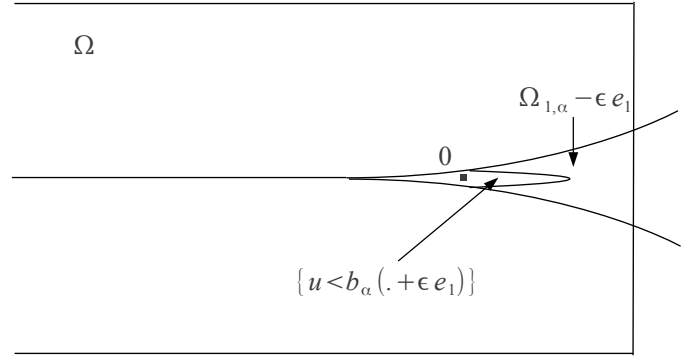}
 \caption{If $u > b_{\alpha}$ on the right edge of $\overline{\Omega_{1,\alpha}} \cap \Omega$, then we get a contradiction by sliding $b_{\alpha}$ to the left.}
 \label{Barrier}
\end{figure}

By convexity, near each point on the $x_1$ axis where $u$ is zero, there is a singularity of the same type as near the origin. 
We can apply the above argument at all such points to complete the proof.
\end{proof}



\section*{Acknowledgments}
This work was supported by NSF grant DMS-1501152. I would like to thank A. Figalli and Y. Jhaveri for helpful comments.



\end{document}